\documentclass{amsart}
\usepackage{amssymb,amsmath,amsthm,amsfonts,epsfig}
\usepackage{paralist}
\usepackage[utf8]{inputenc}
\usepackage[T1]{fontenc}
\usepackage[english]{babel}

\usepackage[colorlinks=true]{hyperref}
\hypersetup{urlcolor=blue, citecolor=red}

\theoremstyle{plain}
\newtheorem{teo}{Theorem}[section]
\newtheorem{cor}[teo]{Corollary}
\newtheorem{lem}[teo]{Lemma}
\newtheorem{prop}[teo]{Proposition}

\theoremstyle{definition}
\newtheorem{df}[teo]{Definition}
\newtheorem{rmk}[teo]{Remark}

\DeclareMathOperator{\length}{len}
\DeclareMathOperator{\lip}{\mathcal{L}}
\DeclareMathOperator{\lips}{lips}
\DeclareMathOperator{\loglips}{loglips}

\DeclareMathOperator{\Diff}{Diff}

\DeclareMathOperator{\dist}{dist} 
\DeclareMathOperator{\distF}{d}

\newcommand{\id}{id}

\newcommand{\fix}{\hbox{Sing}}
\newcommand{\sing}{\fix}

\newcommand{\R}    {\mathbb R}

\newcommand{\C}{\mathcal C}
\newcommand{\Z}  {\mathbb Z}
\newcommand{\N}  {\mathbb N}
\renewcommand{\epsilon}{\varepsilon}

\begin{document}
\title{Lipschitz perturbations of expansive systems}

\author[A. Artigue]{Alfonso Artigue}

\email{artigue@unorte.edu.uy}
\address{Departamento de Matemática y Estadística del Litoral, Universidad de la Rep\'ublica, 
 Gral. Rivera 1350, Salto, Uruguay}

\maketitle
\begin{abstract}
We extend some known results from smooth dynamical systems to 
the category of Lipschitz homeomorphisms of compact metric spaces. 
We consider dynamical properties as robust expansiveness and structural stability allowing Lipschitz perturbations 
with respect to a hyperbolic metric.
\textcolor{black}{We also study the relationship between Lipschitz topologies and the $C^1$ topology on smooth manifolds.}
\end{abstract}

\section{Introduction}
In the study of dynamical systems with discrete time
we can distinguish
two important categories: the \emph{smooth} as for example
$C^1$ diffeomorphisms of smooth manifolds 
and the \emph{topological} as homeomorphisms of metric spaces. 
Here we will consider \emph{Lipschitz} homeomorphisms, an intermediate one. 
Hyperbolicity is a key property of smooth systems.
In this article we will \textcolor{black}{consider Fathi's hyperbolic metric defined for an 
arbitrary expansive homeomorphism of a compact metric space to study its Lipschitz perturbations. }

Recall that a homeomorphism $f\colon X\to X$ of a compact metric space 
is \emph{expansive} if 
there is $\delta>0$ such that $\dist(f^n(x), f^n(y))\leq\delta$ for all $n\in\Z$ implies $x=y$. 
In \cite{Fa89} Fathi proved that for every expansive homeomorphism there are a metric $\distF$, defining the original topology, 
and two parameters $\lambda>1$ and $\epsilon>0$ such that 
if $\distF(x,y)<\epsilon$ then $\distF(f(x),f(y))\geq \lambda\distF(x,y)$ or 
$\distF(f^{-1}(x),f^{-1}(y))\geq \lambda\distF(x,y)$.
\textcolor{black}{Fathi proposed to called such metric as an \emph{adapted hyperbolic metric}.}
The hyperbolic metric is our starting point in the extension of results from hyperbolic diffeomorphisms to 
expansive homeomorphisms in the Lipschitz category. 

A special type of diffeomorphisms with a hyperbolic behavior are the quasi-Anosov diffeomorphisms 
defined by Mañ\'e.
A diffeomorphisms $f\colon M\to M$ of a compact smooth manifold is \emph{quasi-Anosov} if its tangent map is expansive. 
Mañ\'e in \cite{Ma75} proved that the interior in the $C^1$ topology of the set of expansive diffeomorphisms is the set of 
quasi-Anosov diffeomorphisms. 
In particular, every $C^1$ small perturbation of a quasi-Anosov diffeomorphism is expansive, i.e. they are $C^1$ \emph{robustly expansive}.
In Proposition \ref{QALipsSS} we will show that quasi-Anosov diffeomorphisms are robustly expansive even allowing Lipschitz perturbations.

Let us explain how we arrived to the Lipschitz category.
The key is Theorem 6 in Walters paper \cite{Wa78}. 
There it is proved some kind of \emph{structural stability} of transitive subshifts of finite type. 
We will show that Walters perturbations are in fact close in a Lipschitz topology, see Proposition \ref{propEquivMetrics}. 
In particular he showed that these subshifts are robustly expansive. 
In this proof he uses the hyperbolicity of the usual distance for the shift map. 
{Using similar ideas, in} Corollary \ref{expQA} we conclude that every expansive homeomorphism of a compact metric space is 
Lipschitz robustly expansive with respect to a hyperbolic metric.
Another consequence of our results is that Anosov diffeomorphisms are structurally stable even allowing Lipschitz 
perturbations, see Corollary \ref{corAnosov}.
That is, a homeomorphism obtained after a small Lipschitz perturbation is conjugate to the initial Anosov diffeomorphism.
\textcolor{black}{This result extends the one obtained in \cite{Ta74} as explained in Remark \ref{rmkTakaki}.}

Structural stability is related with the shadowing or 
pseudo-orbit tracing property. See for example \cite{Wa78}.
\textcolor{black}{In \cite{PiTi} it is shown that a strong version of the shadowing property (called Lipschitz shadowing) is in fact equivalent with structural stability 
of $C^1$ diffeomorphisms.}
Expansive homeomorphisms of surfaces are known to be conjugate to pseudo-Anosov diffeomorphisms \cite{Le89,Hi90} and 
if the surface is not a torus then pseudo-Anosov maps have not the shadowing property.
But, as proved \textcolor{black}{by Lewowicz} in \cite{Le83}, they are persistent 
and, in some sense, $C^1$ structurally stable. 
See \cite{Le83} for details and precise definitions.
To prove the structural stability 
Lewowicz assumed that the perturbed homeomorphism coincides 
with the original one at the singular points of the stable foliation. See \cite[Theorem 3.5]{Le83}.
We will show that pseudo-Anosov diffeomorphisms are Lipschitz-structurally 
stable with respect to a hyperbolic metric, see Corollary \ref{corPseudoAnosov}. 
In our setting we can also prove that singular points cannot be moved with a small Lipschitz perturbation, i.e., Lewowicz's hypothesis 
holds.
\textcolor{black}{This is done in Theorem \ref{teoSingPA}.}

Let us mention a very special property of the Lipschitz category that is not shared with diffeomorphisms. 
It is known that there are homeomorphisms of smooth manifolds that are not conjugate to smooth diffeomorphisms. 
An example can be obtained from the homeomorphism of $\R^2$ given in polar coordinates by 
$$f(r,\theta)=(r,\theta+\sin(1/r)).$$ See \cite{Ha79} and references therein for more on this subject. 
In the $C^2$ category a well known example is the Denjoy $C^1$ {circle} diffeomorphism (with wandering points) 
that is not conjugate to a $C^2$ diffeomorphism. 
Consequently, from the point of view of topological dynamics, 
there is some loss of generality in assuming that the dynamic is generated by a $C^r$ diffeomorphism.
But, every homeomorphisms, 
even \textcolor{black}{if it is} defined  on a compact metric space,
is conjugate with a bi-Lipschitz homeomorphism. 
Let us prove it. Given a homeomorphism $f\colon X\to X$ of a compact metric space $(X,\dist)$, 
consider a new distance 
$$\distF (x,y)=\sum_{n\in\Z} \frac{\dist(f^n(x),f^n(y))}{2^{|n|}}.$$ 
Now it is easy to see that it defines the original topology and that 
$f$ and $f^{-1}$ are Lipschitz with respect to $\distF$. 
\textcolor{black}{The conjugating homeomorphism is $id:(X,\dist)\to (X,\distF)$.}
See \cite{SaWo89} for more on this subject.
Therefore, from the viewpoint of topological dynamics, 
there is no loss 
of generality
in assuming that the homeomorphism is in fact bi-Lipschitz. 
We wish to present the following question: is every expansive homeomorphism of a compact smooth manifold conjugate to a $C^1$ diffeomorphism?
This seems to be the case in all known examples.

\textcolor{black}{This} paper is organized as follows. 
In Section \ref{secHypMet} we consider hyperbolic metrics for expansive homeomorphisms and we prove its robust 
expansiveness allowing perturbations in the topology considered by Walters in \cite{Wa78}.
In Section \ref{secLips} we introduce the Lipschitz topology in the space of bi-Lipschitz homeomorphisms 
and we prove that it is equivalent with Walters topology. 
\textcolor{black}{We conclude the Lipschitz robust expansiveness of every expansive homeomorphism with respect to a hyperbolic metric. 
In particular, 
quasi-Anosov diffeomorphisms are robustly expansive even allowing Lipschitz perturbations.}
In Section \ref{secSS} we apply \textcolor{black}{techniques} from \cite{Wa78} to obtain the Lipschitz structural stability 
(with respect to a hyperbolic metric) of 
expansive homeomorphisms with the shadowing property. 
With  \textcolor{black}{similar techniques from} \cite{Le83} we obtain 
the same conclusion assuming that the homeomorphism is persistent instead of 
having the shadowing property.
Applications to Anosov and pseudo-Anosov diffeomorphisms are given. 
For pseudo-Anosov homeomorphisms we consider a flat Riemannian metric with conical 
singularities as in, for example, \cite{MaSm93,Je96}.
Using this metric 
we show that singular points of pseudo-Anosov homeomorphisms 
cannot be moved with a small Lipschitz perturbation.
\textcolor{black}{In Section \ref{secC1} we consider smooth manifolds and 
we compare Lipschitz topologies with the usual $C^1$ topology.}


\section{Hyperbolic metrics}
\label{secHypMet}
Let $f\colon X\to X$ be a homeomorphism of a compact metrizable topological space.

\begin{df}
A metric $\dist$ in $X$ defining its topology is $f$-\emph{hyperbolic} if 
there are $\delta>0$ and $\lambda>1$ such that 
if $\dist(x,y)<\delta$ then $$\max\{\dist(f(x),f(y)),\dist(f^{-1}(x),f^{-1}(y))\}\geq \lambda\dist(x,y).$$ 
In this case $\delta$ is an \emph{expansive constant} and $\lambda$ is an \emph{expanding factor} 
and we also say that $f$ is \emph{hyperbolic} with respect to $\dist$.
\end{df}

Our first result, Theorem \ref{teoWalters}, 
is a \textcolor{black}{generalization of \cite[Theorem 6]{Wa78}. }
There it is proved some kind of robust expansiveness of the shift map.
\textcolor{black}{We will consider the topology introduced in \cite{Wa78}.}
Recall that a continuous function $f\colon X\to X$ is \emph{Lipschitz} if there is $k>0$ such that 
$\dist(f(x),f(y))\leq k\dist(x,y)$ for all $x,y\in X$ and a homeomorphisms $f$ is \emph{bi-Lipschitz} if $f$ and $f^{-1}$ are Lipschitz.

\begin{rmk}
\label{fathiMetric}
 In \textcolor{black}{\cite[Theorem 5.1]{Fa89}} Fathi proved that every expansive homeomorphism $f$ defined on a compact metric 
 space admits an $f$-hyperbolic metric defining the \textcolor{black}{original} topology of the space. 
 \textcolor{black}{Moreover, $f$ is a bi-Lipschitz homeomorphism with respect to this hyperbolic metric}.
\end{rmk}

\begin{df}
Given two homeomorphisms $f,g$ of $X$ define the $C^0$-\emph{metric} as 
\begin{equation}
 \label{equC0}
\dist_{C^0}(f,g)=\max_{x\in X} \dist(f(x),g(x))+\max_{x\in X} \dist(f^{-1}(x),g^{-1}(x)).
\end{equation}
Given two Lipschitz homeomorphisms $f,g\colon X\to X$ 
of a compact metric space $(X,\dist)$ consider the following pseudometric
\[
  \dist_W'(f,g)=\sup_{x\neq y}\frac{|\dist(f(x),f(y))-\dist(g(x),g(y))|}{\dist(x,y)}.
 \]
\begin{rmk}
 \textcolor{black}{It is not a metric. In fact it holds that $\dist_W'(f,g)=0$ if and only if $f\circ g^{-1}$ is an isometry of $(X,\dist)$. 
 This is a direct consequence of the definition.}
\end{rmk}

 If $f,g$ are bi-Lipschitz homeomorphisms define the metric
\begin{equation}\label{walterTopo}
  \dist_W(f,g)=\dist_{C^0}(f,g)+\dist_W'(f,g) +\dist_W'(f^{-1},g^{-1}).
 \end{equation}
\end{df}

\begin{rmk}
{ Notice that $\dist_W$ depends on the metric $\dist$ of $X$.}
\end{rmk}

\begin{df}
 A bi-Lipschitz homeomorphism $f\colon X\to X$ {is \emph{$\dist_W$-robustly expansive} if }
 there are $\epsilon,\delta>0$ such that if $\dist_W(f,g)<\epsilon$ then $g$ is expansive with expansive constant $\delta$.
\end{df}

\begin{teo}
\label{teoWalters}
 Every hyperbolic homeomorphism is $\dist_W$-robustly expansive.
\end{teo}

\begin{proof}
Let $f\colon X\to X$ be a hyperbolic homeomorphism with respect to the metric $\dist$ on $X$.
Let $\delta>0$ and $\lambda>1$ be an expansive constant and an expanding factor respectively. 
Take $\epsilon>0$ and $\lambda'>1$ such that $\lambda-\epsilon>\lambda'$. 
We will show that $\lambda'$ is an expanding factor and $\delta$ is an expansive constant 
for $g$ if $\dist_W(f,g)<\epsilon$. 
Take $x,y\in X$ such that $\dist(x,y)\in(0,\delta)$ and assume that $\dist(f(x),f(y))\geq\lambda\dist(x,y)$ 
(the case $\dist(f^{-1}(x),f^{-1}(y))\geq\lambda\dist(x,y)$ is similar). 
Since $\dist_W(f,g)<\epsilon$, we know that 
\[
 \left|\frac{\dist(f(x),f(y))}{\dist(x,y)}-\frac{\dist(g(x),g(y))}{\dist(x,y)}\right|<\epsilon<\lambda-\lambda'.
\]
Then 
\[
 \lambda'\leq\lambda'+\frac{\dist(f(x),f(y))}{\dist(x,y)}-\lambda<\frac{\dist(g(x),g(y))}{\dist(x,y)}
\]
because
\[
 0\leq \frac{\dist(f(x),f(y))}{\dist(x,y)}-\lambda.
\]
Then
$\lambda'\dist(x,y)\leq\dist(g(x),g(y)).$
Therefore $\delta$ is an expansive constant and $\lambda'$ is an expanding factor for $g$.
\end{proof}

\begin{cor}
 \textcolor{black}{For every expansive homeomorphism $f$ of a compact metrizable space $X$ there is a 
 metric $\dist$ in $X$ defining its topology and making $f$ a $\dist_W$-robustly expansive bi-Lipschitz homeomorphism.}
\end{cor}

\begin{proof}
 \textcolor{black}{It follows by \cite[Theorem 5.1]{Fa89} (see Remark \ref{fathiMetric}) and Theorem \ref{teoWalters}.}
\end{proof}

\section{Lipschitz topology}
\label{secLips}

The purpose of this section is to prove that $\dist_W$ induces \textcolor{black}{a} Lipschitz topology in the space of 
bi-Lipschitz homeomorphisms of $X$.

\begin{df}
 Given two bi-Lipschitz homeomorphisms $f,g\colon X\to X$ of a compact metric space $(X,\dist)$ 
 we define 
 \begin{equation}\label{ecuLipsMet}
\dist_L'(f,g)=  \sup_{x\neq y} \left|\log\left(\frac{\dist(f(x),f(y))}{\dist(g(x),g(y))}\right)\right|
 \end{equation}
and the \emph{Lipschitz metric} as
 \[ 
    \dist_L(f,g)= \dist_{C^0}(f,g)+\dist_L'(f,g)+\dist_L'(f^{-1},g^{-1}).
 \]
\end{df}

\begin{rmk}
 \textcolor{black}{As for $\dist_W'$, it holds that $\dist_L'(f,g)=0$ if and only if 
 $f\circ g^{-1}$ is an isometry.}
\end{rmk}

Let us explain why we call it Lipschitz metric. 

\begin{df}
 A \emph{normed group} is a group $G$ with identity $e\in G$ and a function $\|\cdot\|\colon G\to\R$ 
 satisfying:
 \begin{enumerate}
  \item $\|f\|\geq 0$ for all $f\in G$ with equality only at $f=e$,
  \item $\|f\|=\|f^{-1}\|$ for all $f\in G$ and 
  \item $\|fg\|\leq \|f\|+\|g\|$ for all $f,g\in G$.
 \end{enumerate}
\end{df}

Every norm induces a distance as 
$$\dist_{\|\cdot\|}(f,g)=\|fg^{-1}\|+\|g^{-1}f\|.$$
\begin{rmk}
\textcolor{black}{In some sense, we measure how far are $fg^{-1}$ and $g^{-1}f$ from the identity. 
In \cite{WW} it is called as an \emph{operational metric}.}
\end{rmk}
For example, in the group of homeomorphisms of $X$ we can consider the $C^{0}$-\emph{norm}
$$\|f\|_{C^0}=\max_{x\in X} \dist(x,f(x)).$$
And we obtain that $\dist_{\C^0}$ of Equation (\ref{equC0}) is the distance induced by the $C^0$-norm.
\textcolor{black}{See for example \cite{BO} for more on normed groups.}

Let $\lip=\lip(X,\dist)$ denote the group of bi-Lipschitz homeomorphisms of 
a compact metric space $(X,\dist)$. 
The operation in $\lip$ is composition.
As usual, define the \emph{Lipschitz constant} of $f\in\lip$ as 
\[
 \lips(f)=\sup_{x\neq y} \frac{\dist(f(x),f(y))}{\dist(x,y)}.
\]
It is easy to see that:
\begin{enumerate}
 \item $\lips(f)\geq 1$ for all $f\in\lip$ and 
 \item $\lips(f\circ g)\leq \lips(f)\lips(g)$ for all $f,g\in\lip$.
\end{enumerate}
In light of these properties, it is natural to define $\loglips(f)=\log(\lips(f))$
and to consider $$\|f\|_L'=\max\{\loglips(f),\loglips(f^{-1})\}.$$
We have that $(\lip,\|\cdot\|_L')$ is a semi-normed group. 
In fact $\|f\|_L'=0$ if and only if $f$ is an isometry. 
In order to obtain a norm we define $$\|f\|_L=\|f\|_{C^0} + \|f\|_L'.$$

The following proposition explains why we say that $\dist_L$ from Equation (\ref{ecuLipsMet}) is a Lipschitz metric in $\lip$.

\begin{prop}
 It holds that $\dist_{\|\cdot\|_L}=\dist_L$.
\end{prop}

\begin{proof}
By the definitions we have that 
\[
\begin{array}{ll}
 \dist_{\|\cdot\|_L}(f,g) & =\|fg^{-1}\|_L+\|g^{-1}f\|_L\\
                          & =\|fg^{-1}\|_{C^0}+ \|fg^{-1}\|'_L+\|g^{-1}f\|_{C^0}+\|g^{-1}f\|'_L\\
                          & =\dist_{C^0}(f,g)+\|fg^{-1}\|'_L+\|g^{-1}f\|'_L
\end{array}
\]
and
\[
 \dist_L(f,g)=\dist_{C^0}(f,g)+\dist_L'(f,g)+\dist_L'(f^{-1},g^{-1}).
\]
Therefore, it is sufficient to prove that $\dist_L'(f,g)=\|fg^{-1}\|'_L$. 
A similar argument will give us $\dist_L'(f^{-1},g^{-1})=\|g^{-1}f\|'_L$. 
Applying the definitions we have:
\[
\begin{array}{l}
\|fg^{-1}\|'_L  = \max\{\loglips(fg^{-1}),\loglips (gf^{-1})\}\\  
\displaystyle= \max \left\{\sup_{x\neq y}\log\frac{\dist(fg^{-1}(x),fg^{-1}(y))}{\dist(x,y)},\sup_{x\neq y}\log\frac{\dist(gf^{-1}(x),gf^{-1}(y))}{\dist(x,y)}\right\}\\  
\displaystyle= \max \left\{\sup_{x\neq y}\log\frac{\dist(f(x),f(y))}{\dist(g(x),g(y))},\sup_{x\neq y}\log\frac{\dist(g(x),g(y))}{\dist(f(x),f(y))}\right\}\\  
\displaystyle=\sup_{x\neq y} \left|\log\left(\frac{\dist(f(x),f(y))}{\dist(g(x),g(y))}\right)\right|=\dist'_L(f,g).
\end{array}
\]
And the proof ends.
\end{proof}

\begin{prop}
\label{propEquivMetrics}
The metrics $\dist_L$ and $\dist_W$ define the same topology on $\lip$.
\end{prop}

\begin{proof}
For $f\in\lip$ define $$\delta_0=\min\{1/\lips(f),1/\lips(f^{-1})\}$$ 
and $$\delta_1=\max\{\lips(f),\lips(f^{-1})\}.$$
Given $\delta\in(0,\delta_0)$ consider $k>0$ such that 
if $u,v\in [\delta_0-\delta,\delta_1+\delta]$ then 
$$|\log u-\log v|\leq k|u-v|.$$ 
Then, with standard techniques, it can be proved that 
if $\dist_W(f,g)<\delta$ then $\dist_L(f,g)<k\delta$. 
For this purpose, one has to note that
\[
\log\left(\frac{\dist(f(x),f(y))}{\dist(g(x),g(y))}\right)=
\log\left(\frac{\dist(f(x),f(y))}{\dist(x,y)}\right)-\log\left(\frac{\dist(g(x),g(y))}{\dist(x,y)}\right).
\]
The other part of the proof follows by similar arguments.
\end{proof}

\begin{df}
 A bi-Lipschitz homeomorphism $f\in\lip$ is \emph{Lipschitz-robustly expansive} if 
 there are $\epsilon,\delta>0$ such that if $\dist_L(f,g)<\epsilon$ then $g$ is expansive with expansive constant $\delta$.
\end{df}

\begin{cor}
\label{expQA}
 Every hyperbolic homeomorphism is Lipschitz-robustly expansive. 
 Or equivalently, every expansive homeomorphism 
 of a compact metric space
 is Lipschitz robustly expansive with respect to an $f$-hyperbolic metric.
\end{cor}

\begin{proof}
 It is a direct consequence of Theorem \ref{teoWalters} and Proposition \ref{propEquivMetrics}.
\end{proof}

Recall that a $C^1$ diffeomorphism $f\colon M\to M$ is \emph{quasi-Anosov} if the set
$\{\|df^n(v)\|:n\in\Z\}$ is unbounded for all $v\neq0$. 
It is known \cite{Ma75} that quasi-Anosov diffeomorphisms are $C^1$-robustly expansive (considering $C^1$ perturbations).

\begin{prop}
\label{QALipsSS}
 Quasi-Anosov diffeomorphisms are Lipschitz-robustly expansive with respect to every Riemannian metric $\dist_g$.
\end{prop}

\begin{proof}
Let $f$ be a quasi-Anosov diffeomorphism.
By Corollary \ref{expQA} 
 it would be sufficient to prove that $\dist_g$ is hyperbolic for $f$, but this may not 
 be true because the Riemannian metric may not be an adapted metric. 
 Instead, we will show that  $\dist_g$ is hyperbolic for $f^n$ for some $n\geq 1$.
This is sufficient too. 
Consider $n$ such that for all $v\neq 0$, $v\in TM$, it holds that 
$\|df^n(v)\|>2\|v\|$ or $\|df^{-n}(v)\|>2\|v\|$. 
Such $n$ exists, as it is proved in \cite[Lemma 2.3]{Le80} (there it is assumed that $f$ is Anosov but it only uses that $f$ 
is quasi-Anosov).
Now, using the exponential map of the Riemannian metric, we can locally lift $f$ to the tangent fiber bundle. 
In this way we can view $f$ as a small Lipschitz perturbation of $df$. 
Consequently, $\dist_g$ is hyperbolic for $f^n$.
\end{proof}

\section{Lipschitz structural stability}
\label{secSS}
In this section we will consider some shadowing properties to study the structural stability 
allowing Lipschitz perturbations.
\begin{df}
 We say that $f\in \lip(X,\dist)$ is \emph{Lipschitz structurally stable} 
 if {for all $\epsilon>0$ there is $\delta>0$ 
 such that if $\dist_L(f,g)<\delta$ then 
 there is a homeomorphism $h\colon X\to X$ such that 
 $f h=h g$ and $\dist_{C^0}(h,\id)<\epsilon$.}
\end{df}

\begin{df} 
A homeomorphism $f\colon X\to X$ has the \emph{weak shadowing property} 
if for all $\epsilon>0$ there is $\delta>0$ such that 
if $\dist_{C^0}(f,g)<\delta$ then for all $x\in X$ there is $y\in X$ 
such that $\dist(f^n(y),g^n(x))<\epsilon$ for all $n\in\Z$.
\end{df}

\textcolor{black}{In \cite{PiRoSa} the expression \emph{weak shadowing property} is considered with a different meaning.}
Let us recall that a $\delta$-\emph{pseudo-orbit} is a sequence $x_n\in X$ with $n\in\Z$ such that 
$\dist(f(x_n),x_{n+1})<\delta$ for all $n\in\Z$. 
A homeomorphism $f\colon X\to X$ has the (usual) \emph{shadowing property} if 
for all $\epsilon>0$ there is $\delta>0$ such that if $\{x_n\}_{n\in\Z}$ is a $\delta$-pseudo-orbit 
then there is $y\in X$ such that 
$\dist(f^n(y),x_n)<\epsilon$ for all $n\in\Z$. 
This is also called as \emph{pseudo-orbit tracing property}.

\begin{rmk}
\label{rmkShadow}
The shadowing property implies the weak shadowing property. 
Theorems 4,5 and 6 in \cite{Wa78} are stated with the shadowing property but the proofs 
only need the weak shadowing property. {To see that this is true, the reader should check that the proofs of Theorems 
5 and 6 uses the shadowing property via Theorem 4, and in the proof of Theorem 4 it is used, in its first paragraph, 
considering a pseudo-orbit that is a true orbit of the perturbed homeomorphism there called as $S$.}
\end{rmk}

In the general setting of metric spaces, pseudo-orbits may be far from being real orbits of $C^0$ perturbations as we explain 
with the following example.

\begin{rmk}
On compact metric spaces there are homeomorphisms with the weak shadowing property but without the usual shadowing property. 
Let us give an example. Let $f\colon X\to X$ be a homeomorphism of a compact metric space $X$ with three fixed points $a,b,c\in X$. 
Also assume that there are $I\subset X$ and $y\in X$ such that $I$ is a compact arc, $I$ is also an open subset of $X$, $\lim_{n\to-\infty} f^n(I)=a$, 
$\lim_{n\to\infty} f^n(I)=b=\lim_{n\to-\infty} f^n(y)$,
$\lim_{n\to\infty} f^n(y)=c$ and 
$$X=\{a,b,c\}\cup\{f^n(y):n\in \Z\}\cup \bigcup_{n\in \Z}f^n(I).$$ 
See Figure \ref{figExample}. 
\begin{figure}[h]
\begin{center}
 \includegraphics{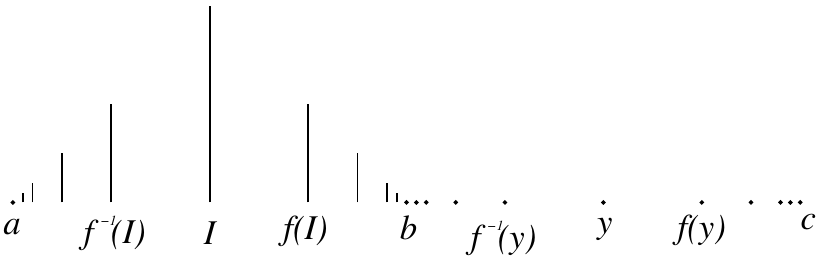}
 \caption{The space $X$ is a countable union of points and intervals.
 The homeomorphism $f$ moves isolated points and intervals to the right, while the accumulation points $a,b,c$ are fixed.}
\label{figExample}
\end{center}
\end{figure}
It is easy to see that $f$ has not the usual shadowing property because there are pseudo-orbits starting at $a$ and 
finishing at $c$ without real orbits tracing it. 
In order to prove that $f$ has the weak shadowing property we have to note that 
every homeomorphism of $X$, in particular a perturbation of $f$, has to fix $a,b,c$ and 
preserve the left part of $b$ in the figure.
\end{rmk}

Let us give another shadowing property that was introduced in \cite{Le83}.

\begin{df}
 A homeomorphism $f\colon X\to X$ is \emph{persistent} if for all $\epsilon>0$ there is $\delta>0$ such that 
 if $\dist_{C^0}(f,g)<\delta$ then for all $x\in X$ there is $y\in X$ such that 
 $\dist(f^n(x),g^n(y))<\epsilon$ for all $n\in\Z$.
\end{df}

Note {the 
difference} between this definition and the weak shadowing property.
{It is known that pseudo-Anosov homeomorphisms of surfaces are persistent but 
they have not the weak shadowing property.}

The proof of the following result uses well known techniques that can be found in \cite{Wa78} 
for the case of the shadowing property. 
For the case of $f$ persistent \textcolor{black}{see \cite[Section 1]{Le83} (Lemma 1.2 and the Remark below it).}

\begin{teo}
\label{structural}
 Let $f\in\lip(M,\dist)$ with $\dist$ an $f$-hyperbolic metric and $M$ a compact manifold without boundary.
 If $f$ is persistent or has the weak shadowing property then $f$ is Lipschitz structurally stable.
\end{teo}

\begin{proof}
Since $\dist$ is $f$-hyperbolic we know by Corollary \ref{expQA} that $f$ is Lipschitz robustly expansive. 
Therefore, there are $\sigma,\alpha>0$ such that if $\dist_L(f,g)<\sigma$ then 
$g$ is expansive with expansive constant $\alpha$.
In order to prove that $f$ is Lipschitz structurally stable take $\epsilon\in(0,\alpha/2)$. 
Since $M$ is a compact manifold without boundary, we can assume that every continuous function $h\colon M\to M$ 
such that $\dist_{C^0}(h,id_M)<\epsilon$ is surjective.
Assume that $f$ has the weak shadowing property. Therefore, there is 
$\delta\in (0,\sigma)$ such that 
if $\dist_{C^0}(f,g)<\delta$ then 
for all $x\in M$ there is $h(x)\in M$ such that 
\begin{equation}
 \label{ecuestest}
\dist(f^n(h(x)),g^n(x))<\epsilon\hbox{ for all }n\in\Z. 
\end{equation}
In this way we have a function $h\colon M\to M$. 
Assume that $\dist_L(f,g)<\delta$.
If $h(x)=h(x')$ then $\dist(f^n(h(x)),g^n(x))<\epsilon$ and $\dist(f^n(h(x')),g^n(x'))<\epsilon$ for all $n\in\Z$. 
Therefore, $\dist(g^n(x),g^n(x'))<\alpha$ for all $n\in\Z$. Since $\alpha$ is an expansive constant for 
$g$ we have that $x=x'$. Consequently, $h$ is injective. 
Let us prove that $h$ is continuous. 
Assume that $x_n\to x$ and $h(x_n)\to y$. 
Taking limit, we have that $\dist(f^n(y),g^n(x))\leq \epsilon$ for all $n\in\Z$. 
Also, we know that $\dist(f^n(h(x)),g^n(x))<\epsilon$. 
Then $\dist(f^n(h(x)),f^n(y))<\alpha$ for all $n\in\Z$.
Since $\alpha$ is an expansive constant for $f$ we conclude that $y=h(x)$. Therefore $h$ is continuous. 
Notice that by (\ref{ecuestest}) we have that $\dist_{C^0}(h,id_M)<\epsilon$ ($n=0$). 
Therefore, $h$ is surjective and a homeomorphism. 
In order to show that $f\circ h=h\circ g$ fix $x\in M$ and note that by (\ref{ecuestest}) it holds that
$$\dist(f^{n+1}(h(x)),g^{n+1}(x))=\dist(f^n(f(h(x))),g^n(g(x)))<\epsilon$$
for all $n\in\Z$. 
\textcolor{black}{Applying (\ref{ecuestest}) again, we have that}
$$\dist(f^n(h(g(x))),g^n(g(x)))<\epsilon$$
for all $n\in\Z$. And by the triangular inequality
$$\dist(f^n(f(h(x))),f^n(h(g(x))))<\alpha$$ 
for all $n\in \Z$. Then $f(h(x))=g(h(x))$ for all $x\in M$ \textcolor{black}{because $\alpha$ is an expansive constant of $f$. }

The case of $f$ persistent is similar, \textcolor{black}{using the condition:
$\dist(f^n(x),g^n(h(x)))<\epsilon$ for all $n\in\Z$, 
instead of (\ref{ecuestest}).}
\end{proof}

\textcolor{black}{Notice that in the previous proof we only used that $M$ is a compact manifold, 
instead of an arbitrary compact metric space, to ensure that small perturbations of the identity are onto.}

\begin{cor}
\label{corAnosov}
Anosov diffeomorphisms are Lipschitz structurally stable with respect to \textcolor{black}{any} Riemannian metric.
\end{cor}

\begin{proof}
 It is well known that Anosov diffeomorphisms have the shadowing property. 
 If the Riemannian metric is not adapted for $f$, we consider an adapted one. 
 \textcolor{black}{Note that all Riemannian metrics are Lipschitz equivalent.}
 In this way the Lipschitz topology on bi-Lipschitz homeomorphisms of $M$ does not change and 
 now the Riemannian metric $\dist_g$ is hyperbolic for $f$.
 Then, by Theorem \ref{structural} the corollary follows.
\end{proof}

\begin{rmk}
\label{rmkTakaki}
\textcolor{black}{
 In \cite{Ta74} a result similar to Corollary \ref{corAnosov} was obtained using a Lipschitz topology defined via local charts. 
 As we will see in Section \ref{secC1} our Lipschitz topology (induced by $\dist_W$) allows more perturbations. 
 Consequently, Corollary \ref{corAnosov} is stronger than the result in \cite{Ta74}.}
\end{rmk}

\begin{cor}
\label{corPseudoAnosov}
Pseudo-Anosov maps of compact surfaces are Lipschitz structurally stable with respect to 
a hyperbolic metric.
\end{cor}

\begin{proof}
 In \cite{Le83} it is shown that Pseudo-Anosov maps are persistent. 
 Therefore we conclude by Theorem \ref{structural}.
\end{proof}

\begin{rmk}
\label{metricPA}
For a pseudo-Anosov map $f$ of a compact surface $S$ denote by $\sing(f)$ the set of singular points of the stable foliation 
of $f$. 
Consider a flat metric in $S\setminus\sing(f)$ such that stable and unstable leaves are geodesics. 
This Riemannian metric induces a $f$-hyperbolic distance in $S$.
The construction is standard and the details can be found for example \textcolor{black}{in \cite[Section 3.1]{Je96}.} 
\textcolor{black}{An important property of this metric, for our purpose, is that the length of the boundary of 
a small ball of radius $r$ centered at 
a singular point $p$ equals $n\pi r$, where $n$ is the number of stable prongs of $p$.}
\end{rmk}

\textcolor{black}{In \cite[Theorem 3.5]{Le83}} Lewowicz obtained a result similar to Corollary \ref{corPseudoAnosov} but in the $C^1$ category. 
In his hypothesis it is required that the perturbed map coincides with the original one at singular points.
\textcolor{black}{In the Lipschitz setting, with a hyperbolic metric, this hypothesis is not needed as the next result shows.}

\begin{teo}
\label{teoSingPA}
 Assume that $f\colon S\to S$ is a pseudo-Anosov homeomorphism of a compact surface 
 with a hyperbolic metric as in Remark \ref{metricPA}.
 Then there is $\epsilon>0$ such that if $\dist_L(f,g)<\epsilon$ then $\sing(g)=\sing(f)$.
\end{teo}

\begin{proof}
If $f$ and $g$ are Lipschitz close then $j=f\circ g^{-1}$ and $j^{-1}$ are Lipschitz close to the identity of $S$.
Then, we can take $\epsilon>0$ such that if $\dist_L(f,g)<\epsilon$ then 
\begin{equation}
\label{eculipsPA}
 \lips(j)\lips(j^{-1})<3/2.
\end{equation}
Let us first show that if $\dist_L(f,g)<\epsilon$ then $j(\sing(f))=\sing(f)$.
Assume by contradiction that there is $x\notin\sing(f)$ but $j(x)\in\sing(f)$
with $\dist_L(f,g)<\epsilon$. 
Consider $C_1$ a small circle of radius $r_1$ and centered at $x$ such that $C_1\cap j(C_1)=\emptyset$. 
Let $r_2=\min_{y\in j(C_1)}\dist(j(x),y)$ and denote by $C_2$ the circle centered at $j(x)$ with radius $r_2$ as shown in 
Figure \ref{figSing}.	
\begin{figure}[h]
\begin{center}
 \includegraphics{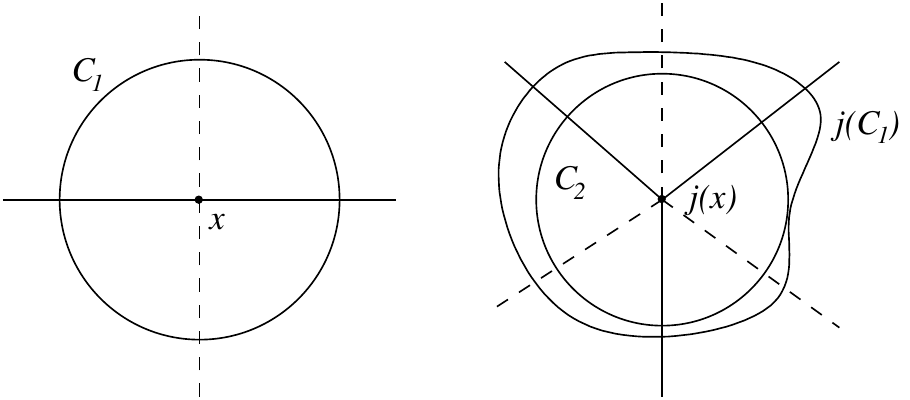}
 \caption{A small Lipschitz perturbation cannot move a singular point.}
\label{figSing}
\end{center}
\end{figure}
Since $x$ is a regular point \textcolor{black}{(of the metric)}
and the metric is flat we have 
that $\length(C_1)=2\pi r_1$, where $\length$ denotes the length of the curve.
Since $j(x)$ is a singular point we have that 
\begin{equation}
\label{ecucompa}
 \length(C_2)\geq 3\pi r_2
\end{equation}
because a neighborhood of $j(x)$ is constructed by gluing at least 3 
flat half-planes \textcolor{black}{as explained in Remark \ref{metricPA}.}
We also have that 
\begin{equation}
 \label{coso2}
 \length(j(C_1))\leq \lips(j)\length(C_1)=\lips(j) 2\pi r_1.
\end{equation}
By the isoperimetric inequality we know that $\length(j(C_1))\geq\length(C_2)$. 
This and (\ref{ecucompa}) implies that 
\begin{equation}
 \label{ecucoso}
\length(j(C_1))\geq 3\pi r_2.
\end{equation}
Finally, by (\ref{coso2}) and (\ref{ecucoso}) we have $3r_2\leq \lips(j) 2r_1$. 
Now take $z\in C_1$ such that $r_2=\dist(j(x),j(z))$. 
Since $z\in C_1$ we have that $\dist(x,z)=r_1$. 
Then 
\[
 \lips(j^{-1})\geq \frac{\dist(x,z)}{\dist(j(x),j(z))}=r_1/r_2.
\]
But, since $r_1/r_2\geq \frac3{2\lips(j)}$ we have that 
$\lips(j)\lips(\textcolor{black}{j}^{-1})\geq 3/2$. This \textcolor{black}{contradicts inequality (\ref{eculipsPA}) and
proves that $j(\sing(f))=\sing(f)$}.

Since $j=fg^{-1}$ and $\sing(f)$ is $f$-invariant we have that $g(\sing(f))=\sing(f)$. 
To conclude the proof we have to note that $\sing(f)$ and $\sing(g)$ are $g$-invariant finite sets and they are close if $\epsilon$ 
is small. Therefore, since $g$ is expansive, we conclude that $\sing(f)=\sing(g)$ as we wanted to prove.
\end{proof}

\section{Topologies on the space of smooth functions}
\label{secC1}

In \cite[page 238]{Wa78} Walters considered $\dist_W$ as ``analogous to a $C^1$ metric''. 
In this section we will investigate the relationship between the $C^1$ topology 
and two different Lipschitz topologies. 
In $\R^n$ we know that every $C^1$ diffeomorphisms is Lipschitz, 
therefore, we can restrict a Lipschitz topology to the space of $C^1$ diffeomorphisms. 
A natural question is: if two diffeomorphisms are close in a Lipschitz metric are they necessarily $C^1$ close? 
This is the problem that will be considered in this section for $\dist_W$, defined at (\ref{walterTopo}) above, 
and the \emph{natural} Lipschitz metric induced by a vector space structure that will be defined below. 
In general, the convergence in the $C^1$ topology implies the convergence in the Lipschitz topology, if we consider 
smooth manifolds with Riemannian metrics. 

For the first result denote by $I=[0,1]$ the unit interval with the usual metric $\dist(x,y)=|y-x|$ for all 
$x,y\in I$. Denote by $\Diff^1(I)$ the group of $C^1$ diffeomorphisms of $I$.
The $C^1$ topology in the one-dimensional case is defined by the distance $\dist_{C^0}(f,g)+\dist_{C^0}(f',g')$ 
for $f,g\in\Diff^1(I)$.
We refer the reader to \cite{Ro} for the definition of the $C^1$ topology on general manifolds. 

\begin{prop}
In $\Diff^1(I)$ the metric $\dist_W$ 
defines the $C^1$ topology. 
\end{prop}

\begin{proof}
If $\dist_{C^0}(f,g)<\epsilon<1/2$ then both diffeomorphisms have derivatives of the same sign. 
Suppose that $f$ and $g$ are increasing, the other case being similar. 
Assume that $\dist_W(f,g)\leq\epsilon$. 
Therefore 
$$\sup_{x<y}\frac{|f(y)-f(x)-g(y)+g(x)|}{y-x}\leq\epsilon.$$ 
Fixing $x$ and taking limit as $y\to x^+$ 
we have that $|f'(x)-g'(x)|\leq\epsilon$. Since this holds for all $x\in I$ we 
have that $f$ and $g$ are $C^1$-close.
\end{proof}

\begin{rmk}
 The argument of the previous proof applied for diffeomorphisms of a compact ball in $\R^n$ will gives us 
 $|\|d_pf(v)\|-\|d_pg(v)\||\leq\epsilon$. But this does not imply that $\|d_pf(v)-d_pg(v)\|\leq\epsilon$. 
\end{rmk}

Let us show that the previous proposition does not hold for 
higher dimensional manifolds. For $\epsilon>0$ given we will construct a diffeomorphism $f\colon D\to D$, 
where $D=\{(x,y)\in \R^2:x^2+y^2\leq 1\}$, so that $\dist'_W(f,id)<\epsilon$ and $d_{(0,0)}f$ is a rotation of angle $\pi$.
For this purpose we need the following preliminary result.
\begin{lem}
\label{lemGamma}
 For all $\epsilon>0$ there are a $C^\infty$ function $\gamma\colon[0,1]\to [0,\pi]$ and $0<a<b<\epsilon$
 such that: 
 \begin{enumerate}
  \item $\gamma(r)=\pi$ for all $r\in [0,a]$,
  \item $\gamma(r)=0$ for all $r\in [b,1]$,
  \item $0\leq |r\gamma'(r)|\leq \epsilon$ for all $r\geq 0$.
 \end{enumerate}
\end{lem}

\begin{proof}
 For $k\in\R$ consider the function $g\colon \R^+\to\R$ given by $g(r)=k-\epsilon\log(r)$. 
 Notice that $g$ satisfies the differential equation $rg'(r)=-\epsilon$ for all $r>0$.
 It is easy to see that if $k<0$ and $|k|$ is sufficiently large then there is $d\in (0,\epsilon)$ such that 
 $g(d)=0$. Fix this value of $k$ and take $c\in (0,d)$ so that $g(c)=\pi$. 
 \begin{figure}[h]
 \begin{center}
  \includegraphics{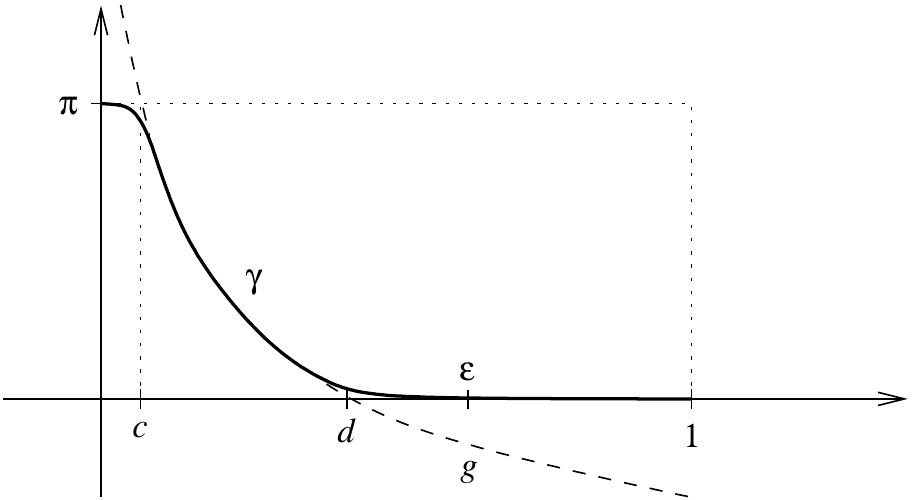}
  \caption{The function $\gamma$ of Lemma \ref{lemGamma}. In dotted line the function $g$.}
  \label{figGamma}
  \end{center}
 \end{figure}
 Now using bump functions near the points $c$ and $d$ it is easy to obtain the function $\gamma$. 
 In Figure \ref{figGamma} the construction is illustrated. 
 The numbers $a$ and $b$ are obtained in $(0,c)$ and $(d,\epsilon)$ respectively.
\end{proof}

Consider $\epsilon>0$ given and the function $\gamma$ given by Lemma \ref{lemGamma}. 
Define in polar coordinates the diffeomorphism $f\colon D\to D$ as 
\[
 f(r\cos\theta,r\sin\theta)=r(\cos(\theta+\gamma(r)),\sin(\theta+\gamma(r))).
\]
Denote by $R_\theta\colon\R^2\to\R^2$ the rotation of angle $\theta$ around the origin. 
Notice that $f(p)=R_\pi(p)$ if $\|p\|\leq a$ and $f(p)=p$ if $\|p\|\geq b$. 
Also, $\|f(p)\|=\|p\|$ for all $p\in D$. 
Therefore, $f$ is $C^0$ close to the identity of $D$ if $\epsilon$ is small 
and for all $\epsilon$ the differential $d_{(0,0)}f=R_\pi$. 
Consequently $f$ is $C^1$ far from the identity. 
We will show that $f$ is close to the identity for the metric $\dist_W$. 
For this, we will first show that the differential of $f$ is close to a rotation at each $p\in D$. 

\begin{lem}
 For all $p\in D$ it holds that $\|d_pf-R_{\gamma(r)}\|\leq \epsilon$ where $r=\|p\|$.
\end{lem}

\begin{proof}
For $\|p\|< a$ or $\|p\|> b$ we know that $d_pf=R_{\gamma(r)}$.
By the symmetry of the construction, it is no loss of generality to consider $p=(r,0)$ with $r>0$. 
 In the standard basis of $\R^2$ the differential of $f$ is given by the matrix 
 \[
  d_pf=\left(
  \begin{array}{l|l}
   \cos(\gamma(r))-r\gamma'(r)\sin(\gamma(r)) & -\sin(\gamma(r))\\
   \sin(\gamma(r))+r\gamma'(r)\cos(\gamma(r)) & \cos(\gamma(r))
  \end{array}
  \right)
 \]
and 
\[
 d_pf-R_{\gamma(r)}=\left(
  \begin{array}{l|l}
   -r\gamma'(r)\sin(\gamma(r)) & 0\\
    r\gamma'(r)\cos(\gamma(r)) & 0
  \end{array}
  \right).
\]
By construction we have that $|r\gamma'(r)|\leq\epsilon$ and then $\|d_pf-R_{\gamma(r)}\|\leq\epsilon$.
\end{proof}

\begin{prop}
 For the diffeomorphism $f$ it holds that $\dist'_W(f,id_D)\leq \epsilon$. 
\end{prop}

\begin{proof}
 Given $p\neq q$ in $D$ consider the line $\alpha(t)=p+t(q-p)$ for $0\leq t\leq 1$. 
 Then 
 \[
 \begin{array}{ll}
  \|f(q)-f(p)\| & \leq \int_0^1 \left\|\frac{d}{dt}f(\alpha(t))\right\| dt\\
  & \leq \int_0^1 (1+\epsilon)\|\frac{d}{dt}\alpha(t)\|dt\\
  & \leq (1+\epsilon)\|q-p\|.
 \end{array}
 \]
Notice that 
$$f^{-1}(r\cos\theta,r\sin\theta)=r(\cos(\theta-\gamma(r)),\sin(\theta-\gamma(r))).$$
Therefore similar estimates gives us 
$$\|q-p\| \leq (1+\epsilon)\|f(q)-f(p)\|$$
and 
\[
 \frac{1}{1+\epsilon}\leq \frac{\|f(q)-f(p)\|}{\|q-p\|}\leq 1+\epsilon
\]
This implies that $\dist'_W(f,id_D)\leq\epsilon$. 
\end{proof}

Therefore we have proved:

\begin{teo}
There is a sequence of $C^1$ diffeomorphisms $f_n$ of the compact disk $D\subset \R^2$ 
do not converging to the identity in the $C^1$ topology but satisfying that $\dist_W(f_n,id_D)\to 0$ as $n\to \infty$.
\end{teo}

Let us now consider another Lipschitz topology.
Suppose that $M\subset \R^n$ is a compact $C^1$ manifold without boundary. 
Denote by $\|\cdot\|$ the Euclidean norm in $\R^n$. 
For $f\colon M\to M$ a bi-Lipschitz homeomorphisms define
\[
 \lips(f)=\sup_{x\neq y}\frac{\|f(x)-f(y)\|}{\|x-y\|}.
\]

\begin{teo}
 The distance 
$$\dist_{C^1}(f,g)=\dist_{C^0}(f,g)+\lips(f-g)$$
defines the $C^1$ topology of $C^1$ diffeomorphisms of $M\subset \R^n$.
\end{teo}

\begin{proof}
 Suppose that $g,f_n\colon M\to M$ are $C^1$ diffeomorphisms, $n\in\N$,
 such that
 $\lim_{n\to\infty}\dist_{C^1}(f_n,g)= 0$.
 Taking local charts we can assume that $g,f_n\colon B_r\subset \R^k\to \R^k$ where $B_r$ is a compact ball and $k=\dim(M)$.
 Now it is easy to see that 
 \[ \|d_pf_n-d_pg\|\leq \lips(f_n-g)\]
 for all $p\in B_r$. Therefore $d_pf_n$ converges to $d_pg$ uniformly in $p\in B_r$.
\end{proof}


\end{document}